\documentclass[a4paper,11pt,preprint]{imsart}

\usepackage[utf8]{inputenc}
\usepackage{amsmath}
\usepackage{type1cm}
\usepackage{amsfonts}
\usepackage{amssymb}
\usepackage{amscd}
\usepackage{amsbsy}
\usepackage{amsthm}
\usepackage{mathrsfs}
\usepackage{graphicx}
\usepackage{color}
\usepackage{mathtools}

\newcommand{\sst}{\scriptstyle}
\makeatletter
\@ifundefined{coloneqq}{%
}{}
\@ifundefined{eqqcolon}{%
\newcommand{\eqqcolon}{\mathrel{=\mkern-1mu\colon\mskip-6mu}}}{}
\def\blfootnote{\xdef\@thefnmark{}\@footnotetext}
\makeatother

\newcommand{\lt}{\left}
\newcommand{\rt}{\right}

\newcommand{\Knn}{K_{n,n}}

\newcommand{\To}{\rightarrow}
\newcommand{\Real}{\mathbf{R}}

\newcommand{\RPlus}{\Real_{+}}

\newcommand{\parl}[1]{\left(#1\right)}

\newcommand{\set}[1]{\left\{#1\right\}}

\newcommand{\card}[1]{\left\lvert#1\right\rvert}

\def\bigO{\mathop{O}\nolimits}
\def\bigOmega{\mathop{\Omega}\nolimits}

\def\bigTheta{\mathop{\Theta}\nolimits}

\newcommand{\limk}{\lim_{k\To\infty}}

\newcommand{\pwit}{\mathcal{T}}

\newcommand{\alphaenumerate}[1]{\let\restoretheenumi\theenumi
\let\restorelabelenumi\labelenumi
\renewcommand{\theenumi}{\alph{enumi}}
\renewcommand{\labelenumi}{(\theenumi)}
#1
\let\theenumi\restoretheenumi
\let\labelenumi\restorelabelenumi
}
\renewcommand{\theenumi}{(\alph{enumi})}
\renewcommand{\labelenumi}{\theenumi}

\makeatletter
\@ifundefined{theorem}{
\@ifundefined{chapter}{ 
\newtheorem{theorem}{\noindent Theorem}
\newtheorem{lemma}{\noindent Lemma}
\newtheorem{corollary}{\noindent Corollary}
\newtheorem{proposition}{\noindent Proposition}
\newtheorem{conjecture}{\noindent Conjecture}
\newtheorem{assertion}{\noindent Assertion}
\newtheorem{assumption}{\noindent Assumption}
\newtheorem{condition}{\noindent Condition}
\newtheorem*{defnn}{\noindent Definition}
}{ 
\newtheorem{theorem}{\noindent Theorem}[chapter]
\newtheorem{lemma}{\noindent Lemma}[chapter]
\newtheorem{corollary}{\noindent Corollary}[chapter]
\newtheorem{proposition}{\noindent Proposition}[chapter]
\newtheorem{conjecture}{\noindent Conjecture}[chapter]
\newtheorem{assertion}{\noindent Assertion}[chapter]
\newtheorem{assumption}{\noindent Assumption}[chapter]
\newtheorem{condition}{\noindent Condition}[chapter]
\newtheorem*{defnn}{\noindent Definition}
}
}
{}
\makeatother

\newcommand{\rd}{\mathrm{d}} 
\newcommand{\eqdist}{\stackrel{\text{\upshape\tiny D}}{=}}  
\chardef\ii="10

\DeclareMathAlphabet{\bi}{OT1}{ptm}{b}{it}
\newcommand{\allcomment}[1]{\relax}

\renewcommand{\rd}{\mathop{}\!\mathrm{d}}

\newcommand{\minorderd}{{\operatorname{min}}^{(d)}}

\renewcommand{\theenumi}{\alph{enumi}}
\renewcommand{\labelenumi}{(\theenumi)}

\makeatletter
\newcommand{\manuallabel}[2]{\def\@currentlabel{#2}\label{#1}}
\makeatother

\usepackage[pdfstartview=FitH,bookmarksopen=true,bookmarksopenlevel=1]{hyperref}

\begin{document}
\begin{frontmatter}

\title{Solutions to recursive distributional equations for the mean-field
TSP and related problems}
\runtitle{Recursive distributional equations for the mean-field TSP}


\author{\fnms{Mustafa} \snm{Khandwawala}\thanksref{t1}\ead[label=e1]{mustafa@ece.iisc.ernet.in}}
\thankstext{t1}{The author is currently at INRIA, Paris, France.}
\runauthor{Khandwawala}

\affiliation{Indian Institute of Science}

\address{Department of Electrical Communication Engineering\\
Indian Institute of Science\\
Bangalore 560012, India\\
\printead{e1}}



\begin{abstract}
For several combinatorial optimization problems over random structures, the theory of \emph{local weak convergence} from probability and the \emph{cavity method} from statistical physics can be used to deduce a recursive equation for the distribution of a quantity of interest. We show that there is a unique solution to such a \emph{recursive distributional equation} (RDE) when the optimization problem is the traveling salesman problem (TSP) or from a related family of minimum weight \emph{$d$-factor} problems (which includes minimum weight matching) on a complete graph (or complete bipartite graph) with independent and identically distributed edge-weights from the exponential distribution. We analyze the dynamics of the operator induced by the RDE on the space of distributions, and prove that the iterates of the operator, starting from any arbitrary distribution, converges to the fixed point solution, modulo shifts. The existence of a solution to the RDE in such a problem 
helps in proving 
results about the limit of the optimal solution of the combinatorial problem. The convergence of the iterates of the operator is important in proving results about the validity of belief propagation for iteratively finding the optimal solution.

\end{abstract}

\begin{keyword}[class=AMS]
\kwd[Primary ]{68Q87} 
\kwd[; secondary ]{60C05} 
\kwd{60E05} 
\end{keyword}

\begin{keyword}
\kwd{belief propagation}
\kwd{combinatorial optimization}
\kwd{fixed point}
\kwd{matching}
\kwd{mean-field}
\kwd{probability distribution}
\kwd{recursive distributional equation}
\kwd{traveling salesman problem}
\end{keyword}

\end{frontmatter}


\section{Introduction} \label{sec:introduction}
We address here a class of fixed point equations over the space of distributions over $\Real$, called \emph{recursive distributional equations} RDEs, that correspond to optimization problems over complete graphs (or complete bipartite graphs) with independent and identically distributed (i.i.d.) edge-weights. A fixed point (which is a distribution over $\Real$) of such an equation contains information about the cost and structure of the solution to the optimization problem, asymptotically as the number of vertices grow to infinity. Furthermore, the dynamics of the operator associated with the fixed point equation corresponds to the dynamics of iterative algorithms like belief propagation (BP) for solving the optimization problem. We prove that each RDE in this class has a unique fixed point, and we also characterize the dynamics of the corresponding operator.

A $d$-factor of a graph is a $d$-regular subgraph of the graph with the same vertex set as the graph. Clearly a 1-factor is a perfect matching. We are interested in the $n\To\infty$ asymptotics of the minimum weight $d$-factor problem on a complete $n\times n$ bipartite graph with i.i.d.\ exponentially distributed edge-weights. Frieze \cite{Fri2004} proved that the cost of the traveling salesman problem (TSP) is asymptotically the same as the cost of the minimum weight 2-factor problem when we have i.i.d.\ weights on the edges. Since we are only interested in asymptotics, both the minimum weight perfect matching and TSP are included in our problems of interest.

\begin{figure}[t]
\def\svgwidth{120px}
\centering
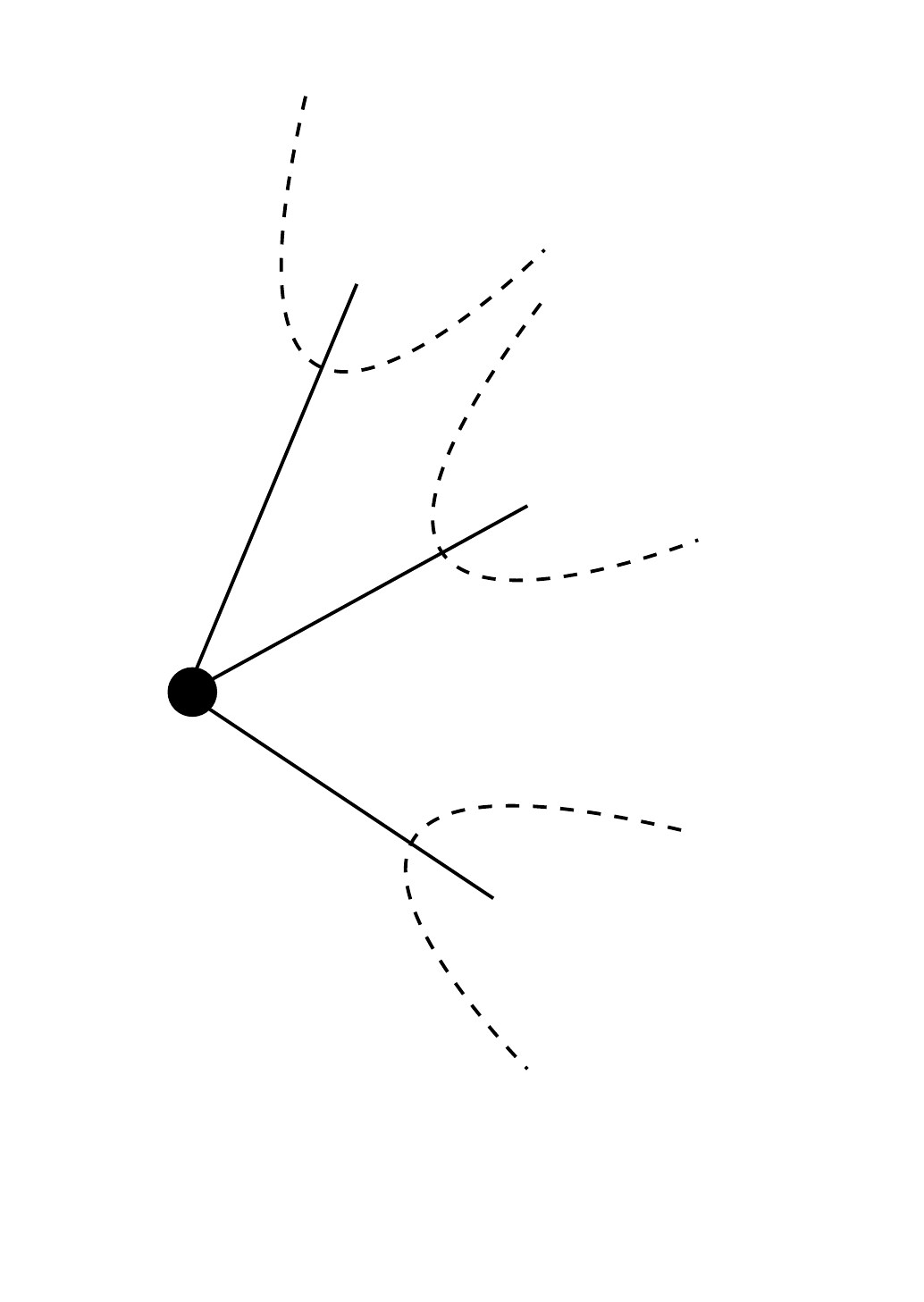
\caption{\label{fig:subtrees}A tree $\pwit$ with the subtrees $\pwit^j$ at node $j$.}
\end{figure}

Let us see with an example the type of recursive distributional equation we get. We use the simplest problem in our ensemble: minimum weight perfect matching problem. Suppose we are interested in solving the problem over a finite tree $\pwit$.
Let $\phi$ be the root of the tree, number its children 1, 2, 3, etc.\ arbitrarily, and call the subtrees rooted at these children $\pwit_1,\pwit_2,\pwit_3$ and so on, as in figure \ref{fig:subtrees}. Denote by $\pwit\setminus\phi$ the graph that remains after removing the root vertex $\phi$ from $\pwit$, that is, the union of the subtrees. Write $C(G)$ for the cost of the minimum weight matching on a graph $G$. It is easy to get the following:
\begin{equation*}
 C(\pwit\setminus\phi)=\sum_{i\sim\phi}C(\pwit_i),
\end{equation*}
and
\begin{equation*}
 C(\pwit)=\min_{i\sim\phi}\set{w(\phi,i)+C(\pwit_i\setminus i)+\sum_{\substack{j\sim\phi\\ j\neq i}}C(\pwit_j)},
\end{equation*}
where $w(\phi,i)$ is the weight of edge $\set{\phi,i}$; in the second equation we optimize over the choice of neighbor $i$ matched to $\phi$. Taking the difference, we have
\begin{equation}\label{eq:costrecursion}
 C(\pwit)-C(\pwit\setminus\phi)=\min_{i\sim\phi}\set{w(\phi,i)-\lt(C(\pwit_i)-C(\pwit_i\setminus i)\rt)}.
\end{equation}

\begin{figure}[tb]
\def\svgwidth{120px}
\centering
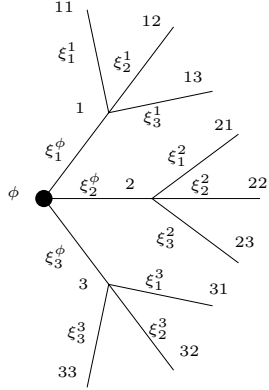
\caption{\label{fig:pwit}PWIT $\pwit$ up to depth 2, with only the first three children of each vertex shown.}
\end{figure}

Aldous \cite{Ald1992,Ald2001} proved that the sequence of complete bipartite graphs $\Knn$ with i.i.d.\ exponentially distributed edge-weights converges weakly to what is called the \emph{Poisson weighted infinite tree} PWIT. Refer to \cite[Section~2]{AldSte2004} for the precise definition of this notion of convergence called \emph{local weak convergence}. For our purpose, it suffices to know that this limit object, PWIT, has a recursive structure as follows. The root $\phi$ has countably infinite children (labeled $1,2,\ldots$) and the corresponding edge weights $\xi^\phi_1,\xi^\phi_2,\ldots$ are points of a rate 1 Poisson process on $[0,\infty)$. This structure is repeated recursively at each new vertex with an independent Poisson process for the weights of the edges connecting it with its children. See figure \ref{fig:pwit}.

This structure of the PWIT allows us to express (\ref{eq:costrecursion}) -- even though it is not defined for an infinite tree -- as an equation among random variables as follows.
\begin{equation*} 
 X \eqdist \min_j (\xi_j-X_j),
\end{equation*}
where $\set{\xi_j,j\ge 1}$ are points of a Poisson process of rate 1 on $\RPlus$, $\set{X_j,j\ge 1}$ are i.i.d.\ random variables, independent of the Poisson process, and having the same distribution as $X$.
Such an equation is called a recursive distributional equation (RDE) (note that it is an equation in distributions). Aldous \cite{Ald2001} showed using elementary computations that the solution to the matching RDE is the logistic distribution
\[
 P(X\le x)=\frac{1}{1+e^{-x}}.
\]
He then used the solution to the RDE to prove that the optimal cost of the matching problem on $\Knn$ converges to $\pi^2/6$. Aldous's method, called the \emph{objective method} is outlined in \cite[Section~7.5]{AldBan2005}, and it involves RDE as a key tool.

RDEs arise in a number of other contexts; see \cite{AldBan2005} for a survey. Such equations also appear from the cavity method in statistical physics \cite{MezPar1986}.

For the minimum weight $d$-factor problem, given an integer $d\ge 1$, the equation takes the following form:
\begin{equation} \label{eq:rdegeneral}
 X \eqdist \minorderd_j (\xi_j-X_j),
\end{equation}
where $\minorderd$ denotes the $d$-th minimum term, $\set{\xi_j,j\ge 1}$ are points of a Poisson process of rate 1 on $\RPlus$, $\set{X_j,j\ge 1}$ are i.i.d.\ random variables, independent of the Poisson process, and having the same distribution as $X$. 

Let the complementary cdf of $X_j$ be $F$. It is an easy exercise to check that the complementary cdf $TF$ of $\minorderd_j (\xi_j-X_j)$ is given by
\begin{equation}\label{eq:Tmap}
TF(x)=\exp\left(-\int_{-x}^\infty F(t)\rd t\right)\left(\sum_{i=0}^{d-1}\frac{\parl{\int_{-x}^\infty F(t)\rd t}^i}{i!}\right),\; x\in\Real.
\end{equation}
Fixed points of the map $T$ are solutions to the RDE (\ref{eq:rdegeneral}).

Parisi and W\"{a}stlund \cite{ParWas2012} proved that the TSP RDE ($d=2$) has a unique solution. W\"{a}stlund \cite{Was2010} showed the limit of the optimal cost of the TSP without relying on an RDE, nevertheless the proof for the solution to the RDE established that the limit constant in \cite{Was2010} is the same as that predicted by replica and cavity methods \cite{MezPar1985,KraMez1989}. Their method can, in principle, be applied to the case of $d>2$ also.

As shown in \cite{SalSha2009,KhaSun2012} the dynamics under the map $T$ is closely related to the updates of belief propagation. Consequently, it is of interest to analyze this dynamics in addition to exploring the fixed points. 
Even when the RDE for a problem on complete graphs cannot be solved explicitly, asserting the existence of a solution directs the construction of an optimal solution on the local weak limit of the graph sequence -- the Poisson weighted infinite tree (PWIT). This leads to a lower bound on the limit optimal cost. Also we can use the uniqueness of the RDE solution, and the convergence to the solution under the iterates of the map $T$ to this solution to prove that belief propagation generates an asymptotically optimal solution. However, to complete these proofs we need one more property related to the RDE called endogeny \cite[Section~2.4]{AldBan2005}, which is often difficult to establish. See \cite{SalSha2009} and \cite{KhaSun2012} for implementation of Aldous's program for the problems of matching and edge cover.

Here we prove that for any integer $d\ge 1$, the RDE (\ref{eq:rdegeneral}) has a unique solution, or equivalently, that the map $T$ has a unique fixed point. In doing so, we give an essential characterization of the domain of attraction of the fixed point. It remains to find whether the associated recursive tree process (RTP) is endogenous for $d\ge 2$. (The RTP for the matching case, $d=1,$ is endogenous \cite{Ban2011}.)

Write $\mathcal{D}$ for the space of complementary cdfs $F$ of proper random variables that satisfy $\int_0^\infty F(t)\rd t<\infty$. The following theorem formalizes our main result.
\begin{theorem}\label{thm:uniquefixedpt}
For any $d\ge 1$, the map $T$ in (\ref{eq:Tmap}) has a unique fixed point $F_d$.

For every distribution $F$ in $\mathcal{D}$, there exists a real $\gamma$ such that for all $x\in\Real$,
\begin{equation}
\begin{split}
\limk T^{2k} F(x) &= F_d(x-\gamma),\\
\limk T^{2k+1} F(x) &= F_d(x+\gamma).
\end{split}
\end{equation}
\end{theorem}

In the next section, we will prove this result using a series of Lemmas. The following simple observation formalized in Lemma~\ref{lem:Tmapboundshrink} is the basis of the approach. If we have real numbers $m\le M$ such that $TF(x-m)\le F(x)\le TF(x-M)$ for all $x\in\Real$ and some \emph{nice} function $F$, then 
$T^3F(x-m)\le T^2F(x)\le T^3F(x-M)$ for all $x\in\Real$. Thus working in increments of two, we have control over the shift between the functions $T^{2k}F$ and $T^{2k+1}F$. Lemma~\ref{lem:Tmapstrict} makes the inequalities strict. To complete the proof of Theorem~\ref{thm:uniquefixedpt}, we will replace $m$ and $M$ with a sequence of pairs $m_k$ and $M_k$, where $m_k \le M_k$, such that both approach a constant $\gamma$ (which depends on $F$) as $k\To\infty$. For convenience, we will work with transforms of the complementary cdfs defined in terms of shifts.

The method we follow here generalizes the method of Salez and Shah \cite{SalSha2009} for the analysis of the iterates generated by the matching RDE (corresponding to $d=1$). Their work uses the known closed-form solution (the logistic distribution) to the matching RDE; we do not have that advantage here.

\section{Proof of Theorem~\ref{thm:uniquefixedpt}} \label{sec:cvgiter}
First we do some bookkeeping. Let $\mathcal{C}_p$ denote the space of differentiable nondecreasing functions from $\Real$ to $\RPlus$ that vanish at $-\infty$ and diverge to $\infty$ at $\infty$. Define a map $I:\mathcal{D}\To\mathcal{C}_p$
\begin{equation*}
IF(x)=\int_{-x}^\infty F(t)\rd t,\;x\in\Real.
\end{equation*}
Define $P:\RPlus\To(0,1]$ $P(y)=e^{-y}\left(\sum_{i=0}^{d-1}\frac{y^i}{i!}\right)$. 
We then have
\begin{equation}\label{eq:TFIF}
TF=e^{-IF}\left(\sum_{i=0}^{d-1}\frac{(IF)^i}{i!}\right)=P(IF).
\end{equation}
Observe that $IF^\prime(x)=F(-x)$. We can write the derivative of $TF$ as
\begin{equation}\label{eq:derivative}
TF^\prime(x)=-e^{-IF(x)}\frac{(IF(x))^{d-1}}{(d-1)!}F(-x).
\end{equation}

\begin{lemma} \label{lem:Freg1}
Let $\mathcal{D}_1\subset\mathcal{D}$ contain all functions in $D$ that are strictly decreasing, $1$-Lipschitz continuous, and differentiable. If $F\in\mathcal{D}$ then $T^k F\in\mathcal{D}_1$ for all $k\ge 2$.
\end{lemma}
\begin{proof}
Since $F$ is nonincreasing, and not identically 0, we can find some $x_0$ and some $\alpha>0$ such that $F(x)\ge \alpha$ for all $x<x_0$. Then
\begin{equation*} 
 IF(x)=\bigOmega(\alpha x)\quad\text{as }x\To\infty,
\end{equation*}
and 
\begin{equation*}
 TF(x)=\bigO\parl{e^{-\alpha x}x^{d-1}}\quad\text{as }x\To\infty,
\end{equation*}
which is integrable on $[0,\infty)$. This implies that $TF\in\mathcal{D}$.

Comparing (\ref{eq:derivative}) and (\ref{eq:TFIF}), we get
\begin{equation} \label{eq:Tderivativebd}
\card{TF^\prime(x)}\le TF(x)\le 1.
\end{equation}
This shows that $TF$ is $1$-Lipschitz. Also, $TF>0$, and so $T^2F$ is strictly decreasing.
\end{proof}

For $F\in\mathcal{D}_1$, define a transform $\widehat{F}$ such that
\begin{equation*}
F(x)=TF\parl{x-\widehat{F}(x)},\;x\in\Real.
\end{equation*}
$\widehat{F}$ denotes the location-dependent shift in $F$ on applying the $T$ map. Since the functions $F$ and $TF$ are monotone and continuous, $\widehat{F}$ is also continuous. Observe that
\begin{equation*}
a<\widehat{F}(x)<b\text{ if and only if }TF(x-a)<F(x)<TF(x-b).
\end{equation*}

\begin{lemma} \label{lem:Tmapboundshrink}
Suppose $F\in\mathcal{D}_1$, and there exist real numbers $m$ and $M$ such that $m\le\widehat{F}(x)\le M$ for all $x\in\Real$. Then $-M\le \widehat{TF}(x)\le -m$ for all $x\in\Real$.
\end{lemma}
\begin{proof}
We have $TF(x-m)\le F(x)\le TF(x-M)$ for all $x\in\Real$. Consequently
\begin{equation*}
\int_{-x}^\infty TF(t-m)\rd t
\le\int_{-x}^\infty F(t)\rd t
\le\int_{-x}^\infty TF(t-M)\rd t
\end{equation*}
for all $x$, and so
\begin{equation} \label{eq:ITFineq1}
ITF(x+m)\le IF(x)\le ITF(x+M).
\end{equation}
Applying the decreasing map $P$ to (\ref{eq:ITFineq1}), we have
\begin{equation*}
T^2F(x+m)\ge TF(x)\ge T^2F(x+M).
\end{equation*}
This implies $-m\ge \widehat{TF}(x)\ge -M$ for all $x$.
\end{proof}

Applying the map $T$ once more gives $m\le\widehat{T^2F}(x)\le M$. Thus, if $\widehat{F}$ is bounded then $\inf \widehat{T^{2k} F}$ is increasing in $k$ and $\sup \widehat{T^{2k} F}$ is decreasing in $k$, and so they converge to some $m^*$ and $M^*$ respectively ($m^*\le M^*$):
\begin{equation} \label{eq:Tmaplimits}
 \inf \widehat{T^{2k} F}\uparrow m^* \text{ and } \sup \widehat{T^{2k} F} \downarrow M^* \text{ as } k\To\infty.
\end{equation}
We will show that $m^*=M^*$. 

We first show that the boundedness assumption is satisfied after a few iterations.
\begin{lemma} \label{lem:hatbounded}
 For any $F\in\mathcal{D}_1$, $\widehat{T^4 F}$ is bounded.
\end{lemma}
\begin{proof}
We have
\begin{align*}
 P(y)&=\bigTheta\parl{y^{d-1}e^{-y}}&&\text{as }y\To\infty,\\
 P(y)&=1-\bigTheta\parl{y^d}&&\text{as }y\To 0.
\end{align*}
Since $\int_0^\infty F<\infty$ for any $F\in \mathcal{D}_1$, we can write
\begin{align}
 TF(x)&=\bigTheta\parl{\parl{\int_{-x}^{x_0}F}^{d-1} e^{-\int_{-x}^{x_0}F}}&&\text{as }x\To\infty, \label{eq:order1}\\
 TF(x)&=1-\bigTheta\parl{\parl{IF(x)}^d}=1-\bigTheta\parl{\parl{\int_{-x}^{\infty}F}^d}&&\text{as }x\To -\infty \label{eq:order2}.
\end{align}
for any $x_0$. The above equations hold with $T^k F$ on the left-hand side and $T^{k-1}F$ on the right-hand side, for $k\ge 1$.

As in Lemma~\ref{lem:Freg1}, we can find $\alpha>0$ such that $1\ge F(x)\ge \alpha$ for all sufficiently small $x$. Then by (\ref{eq:order1}), as $x\To\infty$
\begin{equation} \label{eq:ordera}
 TF(x)=\bigO\parl{x^{d-1}e^{-\alpha x}}\quad\text{and}\quad TF(x)=\bigOmega\parl{x^{d-1}e^{-x}}.
\end{equation}
As $x\To-\infty$, from (\ref{eq:ordera}) we get 
\begin{equation*}
 ITF(x)=\bigO\parl{\card{x}^{d-1}e^{\alpha x}}\quad\text{and}\quad ITF(x)=\bigOmega\parl{\card{x}^{d-1}e^{x}}.
\end{equation*}
Together with \ref{eq:order2}, as $x\To-\infty$, we get
\begin{equation} \label{eq:orderb}
 T^2 F(x)=1-\bigO\parl{\card{x}^{d(d-1)}e^{\alpha d x}}\quad\text{and}\quad T^2 F(x)=1-\bigOmega\parl{\card{x}^{d(d-1)}e^{d x}}.
\end{equation}

Again, using both bounds in (\ref{eq:orderb}) and using (\ref{eq:order1}), as $x\To\infty$
\begin{equation} \label{eq:orderc}
 T^3 F(x)=\bigTheta\parl{x^{d-1}e^{-x}}.
\end{equation}
Compare (\ref{eq:ordera}) and (\ref{eq:orderc}) to see that we have been able to tighten the upper bound in $T^3F$, as $x\To\infty$. Following the steps leading from (\ref{eq:ordera}) to (\ref{eq:orderb}), we now get, as $x\To-\infty$,
\begin{equation*}
 T^4 F(x)=1-\bigTheta\parl{\card{x}^{d(d-1)}e^{d x}}.
\end{equation*}

Repeating this argument, we inductively have that for all $k\ge 4$
\begin{align}
 T^k F(x)&=\bigTheta\parl{x^{d-1}e^{-x}},&&\text{as }x\To\infty\label{eq:finalorder1}\\
 T^k F(x)&=1-\bigTheta\parl{\card{x}^{d(d-1)}e^{d x}}&&\text{as }x\To-\infty\label{eq:finalorder2},
\end{align}
with the constants possibly depending on $k$.

The asymptotic behavior above for $k=4$ asserts the existence of $K>0$ such that $T^5 F(x+K)\le T^4 F(x)\le T^5 F(x-K)$ for all sufficiently large $x$, as well as for all sufficiently small $x$, and hence $\card{\widehat{T^4 F}(x)}\le K$ for such $x$. By continuity of $\widehat{T^4 F}(x),$ this function is bounded over $\Real$.
\end{proof}

Now we show that the terms $\inf \widehat{T^{2k} F}$ and $\sup \widehat{T^{2k} F}$ are strictly monotone in $k$ unless $\widehat{T^{2k} F}$ is constant.
\begin{lemma}\label{lem:Tmapstrict}
Suppose $F=T^4G$ for some $G\in\mathcal{D}_1$. If $\widehat{F}$ is not constant then $\sup \widehat{T^2 F} < \sup \widehat{F}$ and $\inf \widehat{T^2 F} > \inf \widehat{F}$.
\end{lemma}
\begin{proof}
By Lemma~\ref{lem:hatbounded}, $\widehat{F}$ is bounded, and $F$ follows the asymptotics (\ref{eq:finalorder1}) and (\ref{eq:finalorder2}).
Let $M=\sup \widehat{F}$. Since $\widehat{F}$ is not constant, by continuity, there exists an interval $(a,b)$ such that $\sup_{(a,b)}\widehat{F}\eqqcolon M^\prime<M$. We have $F(x)\le TF(x-M)$ for all $x\in\Real$, and $F(x)\le TF(x-M^\prime)$ for all $x\in(a,b)$. 

For $x\ge -a$,
\begin{align*}
\int_{-x}^\infty F(t)\rd t&=\int_{-x}^a F(t)\rd t+\int_{a}^b F(t)\rd t+\int_{b}^\infty F(t)\rd t\\
&\le\int_{-x}^a TF(t-M)\rd t+\int_{a}^b TF(t-M^\prime)\rd t+\int_{b}^\infty TF(t-M)\rd t\\
&=\int_{-x}^\infty TF(t-M)\rd t-\int_{a}^b(TF(t-M) - TF(t-M^\prime))\rd t.
\end{align*}
This implies
\begin{equation} \label{eq:TfromIineq1}
IF(x) \le ITF(x+M)-\kappa,
\end{equation}
where $\kappa>0$.
Since $P$ is strictly decreasing, we have $TF(x)>T^2 F(x+M)$, and also
\begin{align*}
TF(x)&\ge e^\kappa e^{-ITF(x+M)} \left(\sum_{i=0}^{d-1}\frac{(ITF(x+M)-\kappa)^i}{i!}\right)\\
&=e^\kappa T^2 F(x+M) \frac{\sum_{i=0}^{d-1}\frac{(ITF(x+M)-\kappa)^i}{i!}} {\sum_{i=0}^{d-1}\frac{(ITF(x+M))^i}{i!}}
\end{align*}
Choose $x$ sufficiently large so that $ITF(x+M)$ is large enough for the ratio of the sums above to be at least $e^{-\kappa/2}$. So we get 
\begin{equation} \label{eq:TfromIineq2}
TF(x) \ge \kappa_1 T^2 F(x+M),
\end{equation}
for $x$ larger than some $x_0$, with $\kappa_1=e^{\kappa/2}>1$.

Now, for $x\ge -x_0$
\begin{align}
 ITF(x)&=\int_{-x}^\infty TF(t)\rd t=\int_{-x}^{x_0} TF(t)\rd t+\int_{x_0}^{\infty} TF(t)\rd t\notag\\
 &\ge\int_{-x}^{x_0} T^2 F(t+M)\rd t+\int_{x_0}^{\infty} \kappa_1 T^2 F(t+M)\rd t\notag\\
 &=\int_{-x}^{\infty} T^2 F(t+M)\rd t+\int_{x_0}^{\infty} (\kappa_1-1) T^2 F(t+M)\rd t\notag\\
 &=IT^2F(x-M)+\kappa_2,\label{eq:Tstrictineq1}
\end{align}
where $\kappa_2>0$.
For $x< -x_0$
\begin{align}
 ITF(x)&=\int_{-x}^\infty TF(t)\rd t\ge\int_{-x}^{\infty} \kappa_1 T^2 F(t+M)\rd t\notag\\
 &=\kappa_1 IT^2F(x-M).\label{eq:Tstrictineq2}
\end{align}
Inequalities (\ref{eq:Tstrictineq1}) and (\ref{eq:Tstrictineq2}) directly give $T^2 F(x)<T^3 F(x-M)$ for all $x$, and so $\widehat{T^2 F}(x)<M$ for all $x$. We now have to establish the strict inequality in the limit as $x\To\infty$ and $x\To-\infty$.

When $x\To-\infty$, $IG(x)$ approaches 0 for any $G\in\mathcal{D}_1$. The $P$ function satisfies
\begin{equation*}
 P(y)=1-\frac{y^d}{d!}+\bigO\parl{y^{d+1}},\quad\text{as }y\To 0.
\end{equation*}
Therefore, for a fixed $k$,
\begin{equation*}
 P(ky)=P(y)-\frac{y^d}{d!}(k^d-1)+\bigO\parl{y^{d+1}},\quad\text{as }y\To 0.
\end{equation*}
By (\ref{eq:Tstrictineq2}), as $x\To -\infty$,
\begin{equation*}
 T^2 F(x)\le T^3 F(x-M)-\frac{\parl{IT^2F(x-M)}^d}{d!}(\kappa_1^d - 1)+\bigO\parl{\parl{IT^2F(x-M)}^{d+1}}.
\end{equation*}
By (\ref{eq:finalorder1}), $IT^2 F(x)=\bigTheta\parl{\card{x}^{d-1}e^x}$ for $x\To -\infty$. Substituting this in the above equation, we get
\begin{equation*}
 T^2 F(x)\le T^3 F(x-M)-c_1 \card{x}^{d(d-1)}e^{dx}+c_2 \card{x}^{(d+1)(d-1)}e^{(d+1)x},
\end{equation*}
where $c_1>0$, and $x$ is sufficiently small.
Now 
\begin{align*}
\widehat{T^2 F}(x)&=x-\parl{T^3 F}^{-1}\parl{T^2 F(x)} \\
&\le x-\parl{T^3 F}^{-1}\parl{T^3 F(x-M)-c_1 \card{x}^{d(d-1)}e^{dx}+c_2 \card{x}^{(d+1)(d-1)}e^{(d+1)x}},\\
\intertext{which by Taylor's expansion}
 &= x - \parl{ (x-M) - \frac{c_1 \card{x}^{d(d-1)}e^{dx}-c_2 \card{x}^{(d+1)(d-1)}e^{(d+1)x}}{T^3F^\prime(x-M)} + o(1) }
\end{align*}
as $x\To-\infty$.
Using (\ref{eq:finalorder1}), (\ref{eq:finalorder2}) with (\ref{eq:derivative}), gets us the bound $-T^3F^\prime(x-M)=\bigO(\card{x}^{d(d-1)}e^{dx})$ as $x\To-\infty$. So we get
\begin{align*}
 \widehat{T^2 F}(x) \le M - (\kappa_3 -c_3 \card{x}^{d-1}e^{x} - o(1))
\end{align*}
as $x\To-\infty$.
where $\kappa_3>0$. Thus
\begin{equation*}
 \limsup_{x\To-\infty}\widehat{T^2 F}(x)\le M - \kappa_3 < M.
\end{equation*}

We use the steps that give (\ref{eq:TfromIineq2}) from (\ref{eq:TfromIineq1}) to derive the following from (\ref{eq:Tstrictineq1}):
\begin{equation*} 
T^2 F(x) \le \kappa_4 T^3 F(x-M),
\end{equation*}
for $x$ larger than some $x_1$, and $\kappa_4=e^{-\kappa_1/2}<1$.

Now 
\begin{align*}
\widehat{T^2 F}(x)&=x-\parl{T^3 F}^{-1}\parl{T^2 F(x)} \\
&\le x-\parl{T^3 F}^{-1}\parl{\kappa_4 T^3 F(x-M)}\\
&= x-(T^3 F)^{-1}\parl{T^3 F(x-M)-(1-\kappa_4) T^3 F(x-M)}\\
&=x-\parl{x-M-\frac{(1-\kappa_4) T^3 F(x-M)}{\parl{T^3 F}^\prime\parl{\parl{T^3 F}^{-1}(\tau)}}},
\end{align*}
the last equality by Taylor's expansion; $\tau\in[\kappa_4 T^3 F(x-M),T^3 F(x-M)]$.

Using the inequality (\ref{eq:Tderivativebd}), we get
\begin{align*}
\widehat{T^2 F}(x)&\le M-\frac{(1-\kappa_4) T^3 F(x-M)}{\tau}\\
&\le M-(1-\kappa_4),
\end{align*}
for $x>x_1$.
This gives
\begin{equation*}
 \limsup_{x\To\infty}\widehat{T^2 F}(x)\le M - (1-\kappa_4) < M.
\end{equation*}
Consequently, $\sup \widehat{T^2 F} < M$.

The proof for the infimum is along the same lines, and we omit the details.
\end{proof}

Lemma~\ref{lem:hatbounded} and Lemma~\ref{lem:Tmapboundshrink} imply that $\widehat{T^k F}$ is uniformly bounded for all $k\ge 4$: there exists $K>0$ such that
\begin{equation}\label{eq:allhatunifbd}
\card{ \widehat{T^k F}}\le K,
\end{equation}
and hence
\begin{equation}\label{eq:shiftunifbd}
T^{k+1} F(x+K)\le T^k F(x)\le T^{k+1} F(x-K),
\end{equation}
for all $x$.

$\widehat{T^k F}$ is the shift to $T^{k+1} F$ to match it to $T^k F$. The following Lemma bounds the shifts of the entire family $\set{T^k F}_{k\ge 4}$.
\begin{lemma}\label{lem:alltailuniformbd}
 Suppose $F\in\mathcal{D}_1$ is such that $\widehat{F}$ is bounded. Then there exists $\rho>0$ such that
 \begin{equation}\label{eq:alltailuniformbd}
  F(x+\rho) \le T^k F(x) \le F(x-\rho)
 \end{equation}
for all $x\in\Real$ and all $k\ge 4$.
\end{lemma}
\begin{proof}
 Let $G=T^4 F$. By (\ref{eq:finalorder1}), there exists $c>0$ such that
 \begin{equation} \label{eq:firsttailbd}
  G(x) \le cx^{d-1}e^{-x}
 \end{equation}
 for all $x$ larger than some $M$.
 For $x\le -M$,
 \begin{equation*} 
  IG(x) =\int_{-x}^\infty G(t)\rd t\le b_1 c \card{x}^{d-1}e^{x},
 \end{equation*}
 with some $b_1>0$,
and
 \begin{equation} \label{eq:secondtailbd}
  TG(x)=P(IG(x)) \ge 1-b_2 b_1^d c^d \card{x}^{d(d-1)}e^{dx},
 \end{equation}
 with some $b_2>0$.
Then for $x\ge M$,
 \begin{align} 
  ITG(x) &=\int_{-x}^\infty TG(t)\rd t\\
&\ge \int_{-x}^{-M}(1-b_2 b_1^d c^d \card{t}^{d(d-1)}e^{dt})\rd t+\int_{-M}^\infty TG(t)\rd t\\
&\ge x-M-b_3 b_2 b_1^d c^d M^{d(d-1)}e^{-dM}+\int_{-M}^\infty TG(t)\rd t, \label{eq:continueineq}
 \end{align}
with some $b_3>0$.

Now
\begin{align*}
\int_{-M}^\infty TG(t)\rd t = P^{-1}\parl{T^2 G(M)}.
\end{align*}
By applying (\ref{eq:shiftunifbd}) twice, we get
\begin{equation*}
T^2 G(M)\le G(M-2K).
\end{equation*}
Using (\ref{eq:Tderivativebd}) we have
\begin{equation*}
G(M-2K)\le e^{2K}G(M) \le e^{2K} c M^{d-1} e^{-M},
\end{equation*}
where the last inequality comes from (\ref{eq:firsttailbd}).
We can find a constant $b>0$ such that
\begin{align*}
\int_{-M}^\infty TG(t)\rd t \ge M-b.
\end{align*}
Now the inequality (\ref{eq:continueineq}) becomes
\begin{align*}
ITG(x)\ge x-b-b_3 b_2 b_1^d c^d M^{d(d-1)}e^{-dM},
\end{align*}
for $x\ge M$.
We can make the term on the right large such that
\begin{align*}
T^2G(x) &\le b_4(x-b-b_3 b_2 b_1^d c^d M^{d(d-1)}e^{-dM})^{d-1}e^{-(x-b-b_3 b_2 b_1^d c^d M^{d(d-1)}e^{-dM})}\\
&\le b_4 x^{d-1} e^{-x} e^{b+b_3 b_2 b_1^d c^d M^{d(d-1)}e^{-dM}},
\end{align*}
with some $b_4>0$, for all $x\ge M$.
If we take $c>b_4 e^b$, then for $M$ suitably large, the constant above can be made smaller than $c$, and so we have
\begin{equation*}
T^2G(x)\le c x^{d-1}e^{-x},
\end{equation*}
for all $x\ge M$,
whenever the same holds for $G$. Note that none of the constants other than $c$ depend on the starting distribution $G$. Hence, by induction, this inequality holds for all $T^{2k}G$, $k\ge 0$. The inequality (\ref{eq:secondtailbd}) then holds for all $T^{2k+1}G$, $k\ge 0$. Repeating the argument starting with $TG$ instead of $G$, we see that all the tail bounds, both at $+\infty$ and $-\infty$, can be written with a common constant, and so (\ref{eq:alltailuniformbd}) holds with some $\rho>0$.
\end{proof}

The following Lemma asserts that the functions $\widehat{T^kF}$ are essentially constant outside a compact set.
\begin{lemma} \label{lem:hatuniformcontrol}
For any $F\in\mathcal{D}_1$,
\begin{gather*}
 \sup_{k\ge 5}\sup_{x>M} \card{\widehat{T^kF}(x)-\widehat{T^kF}(M)} \xrightarrow{M\To\infty} 0\\ 
\sup_{k\ge 6} \sup_{x>M} \card{\widehat{T^kF}(-x)-\widehat{T^kF}(-M)} \xrightarrow{M\To\infty} 0.
\end{gather*}
\end{lemma}
\begin{proof}
Take $G=T^kF$ for some $k\ge 4$. We have by definition
\begin{equation*}
TG(M)=T^2G\parl{M-\widehat{TG}(M)}, 
\end{equation*}
which implies
\begin{equation}\label{eq:intdiff1}
\int_{-M}^\infty G = \int_{-M+\widehat{TG}(M)}^\infty TG.
\end{equation}
Fix $\epsilon>0$. For sufficiently large $M$, and for $x>M$, consider the difference
\begin{align}
\int_{-x}^{-M}G&-\int_{-x+\widehat{TG}(M)-\epsilon}^{-M+\widehat{TG}(M)}TG&&\notag\\
&\le \int_{-x}^{-M}G - \int_{-x+K-\epsilon}^{-M+K}TG &&\text{by (\ref{eq:allhatunifbd})}\notag\\
&\le \int_{-x}^{-M}1 - \int_{-x+K-\epsilon}^{-M+K}\parl{1-\bigTheta\parl{\card{t}^{d(d-1)}e^{dt}}\rd t} &&\text{by (\ref{eq:finalorder2})}\notag\\
&\le (x-M) - (x-M+\epsilon) \int_{-\infty}^{-M+K}\bigTheta\parl{\card{t}^{d(d-1)}e^{dt}}\rd t && \notag\\
&\le (x-M) - \parl{x-M+\epsilon-c (M-K)^{d(d-1)}e^{-d(M-K)}}&&\notag\\
&\le -\epsilon+c(M-K)^{d(d-1)}e^{-d(M-K)}.&&\label{eq:intdiff2}
\end{align}
where, by Lemma~\ref{lem:alltailuniformbd} and (\ref{eq:allhatunifbd}), $c$ and $K$ do not depend on $G$. The above difference is negative for $M$ sufficiently large. Adding (\ref{eq:intdiff1}) and (\ref{eq:intdiff2}), and using that $P$ is decreasing and the fact that the right-hand side of (\ref{eq:intdiff2}) is negative, we have
\begin{equation*}
TG(x) \ge T^2G\parl{x-\widehat{TG}(M)+\epsilon},
\end{equation*}
which means
\begin{equation} \label{eq:firstpartlowerbd}
\widehat{TG}(x)\ge \widehat{TG}(M)-\epsilon.
\end{equation}
Similar calculation by using (\ref{eq:finalorder2}) on $G$ this time shows that 
\begin{equation} \label{eq:firstpartupperbd}
\widehat{TG}(x)\le \widehat{TG}(M)+\epsilon.
\end{equation}
This holds uniformly for all $T^kF$, $k\ge 5$, which establishes the first part of the Lemma.

We now show the second part. Observe that by definition 
\begin{equation*}
T^2G(-x)=T^3G\parl{-x-\widehat{T^2G}(-x)}, 
\end{equation*}
and so by applying $P^{-1}$, we get
\begin{equation}\label{eq:equalityintegrals1}
ITG(-x)=\int_{x}^\infty TG(t)\rd t=\int_{x+\widehat{T^2G}(-x)}^\infty T^2G(t)\rd t=IT^2G\parl{-x-\widehat{T^2G}(-x)}.
\end{equation}
Using (\ref{eq:firstpartlowerbd}) and (\ref{eq:firstpartupperbd}), we have, for all $t\ge M$,
\begin{equation*}
T^2G(t-\widehat{TG}(M)+\epsilon) \le TG(t) \le T^2G(t-\widehat{TG}(M)-\epsilon),
\end{equation*}
and so 
\begin{equation}\label{eq:inequalityintegrals1}
\int_{x}^\infty T^2G(t-\widehat{TG}(M)+\epsilon)\rd t \le \int_{x}^\infty TG(t)\rd t \le \int_{x}^\infty T^2G(t-\widehat{TG}(M)-\epsilon)\rd t,
\end{equation}
when $x>M$.
Using the equality (\ref{eq:equalityintegrals1}) in (\ref{eq:inequalityintegrals1}), we have
\begin{equation*}
\int_{x}^\infty T^2G(t-\widehat{TG}(M)+\epsilon)\rd t \le \int_{x+\widehat{T^2G}(-x)}^\infty T^2G(t)\rd t \le \int_{x}^\infty T^2G(t-\widehat{TG}(M)-\epsilon)\rd t,
\end{equation*}
i.e.,
\begin{equation*}
\int_{x-\widehat{TG}(M)+\epsilon}^\infty T^2G(t)\rd t \le \int_{x+\widehat{T^2G}(-x)}^\infty T^2G(t)\rd t \le \int_{x-\widehat{TG}(M)-\epsilon}^\infty T^2G(t)\rd t.
\end{equation*}
This gives
\begin{equation*}
-\widehat{TG}(M)+\epsilon \ge \widehat{T^2G}(-x) \ge \widehat{TG}(M)-\epsilon.
\end{equation*}
Substituting $x$ by $M$,
\begin{equation*}
-\widehat{TG}(M)+\epsilon \ge \widehat{T^2G}(-M) \ge \widehat{TG}(M)-\epsilon.
\end{equation*}
Taking the difference, we get
\begin{equation*}
2\epsilon \ge \widehat{T^2G}(-x)-\widehat{T^2G}(-M) \ge -2\epsilon.
\end{equation*}
This proves the second part of the Lemma.
\end{proof}

We now have all the ingredients for the proof of Theorem~\ref{thm:uniquefixedpt}.
\begin{proof}[Proof of Theorem~\ref{thm:uniquefixedpt}]
Take $G=T^6 F$. By Lemmas~\ref{lem:Freg1} and \ref{lem:hatbounded}, for all $k\ge 0$, $\widehat{T^kG}$ are uniformly bounded, and 
\begin{equation*}
 \inf \widehat{T^{2k}G} \uparrow m^*\quad\text{and}\quad \sup \widehat{T^{2k}G} \downarrow M^*.
\end{equation*}
By Lemma~\ref{lem:alltailuniformbd},
$T^kG$ lie within the set
 \begin{equation}\label{eq:setforallcdfs}
 \set{G_1\in\mathcal{D}_1 \mid G(x+\rho) \le G_1(x) \le G(x-\rho) \text{ for all }x}.
 \end{equation}
 
The functions $T^{2k}G,k\ge 0$ are bounded and 1-Lipschitz. By Arzela-Ascoli theorem this sequence is relatively compact with respect to compact convergence. There exists a subsequence that converges uniformly over compacts:
\begin{equation}\label{eq:Titercvg}
 T^{2\vartheta(k)}G\xrightarrow{k\To\infty}G_\infty.
\end{equation}
The restriction to the set (\ref{eq:setforallcdfs}) allows us to use the dominated convergence theorem along with Lipschitz continuity of $P$ to conclude that
\begin{equation*}
 T^{2\vartheta(k)+1}G\xrightarrow{k\To\infty}TG_\infty.
\end{equation*}
Since these functions are continuous and monotone, and restricted to (\ref{eq:setforallcdfs}), the transforms $\widehat{T^{2\vartheta(k)}G}$ also converge uniformly over compact sets, i.e,
\begin{equation*}
 \widehat{T^{2\vartheta(k)}G}\xrightarrow{k\To\infty}\widehat{G_\infty}.
\end{equation*}
The uniform tail behavior from Lemma~\ref{lem:hatuniformcontrol} makes this convergence uniform over all $\Real$. This gives
\begin{align*}
 \inf \widehat{G_\infty} &=\limk \inf \widehat{T^{2\vartheta(k)}G} =m^*,\\
 \sup \widehat{G_\infty} &=\limk \sup \widehat{T^{2\vartheta(k)}G} =M^*.
\end{align*}

Continuing the same arguments with further iterations of the $T$ map on the sequence $T^{2\vartheta(k)}G$, we get
\begin{align*}
 \inf \widehat{T^2G_\infty }&=m^*,\\
 \sup \widehat{T^2G_\infty }&=M^*.
\end{align*}
Lemma~\ref{lem:Tmapstrict} allows this only if $\widehat{G_\infty}$ is constant, say $\gamma$ ($=m^*=M^*$). Note that $\gamma$ depends only on the starting distribution $G$ (and via $G$ on $F$). Writing $\widehat{G_\infty}(x)=\gamma,~x\in\Real$, we get
\begin{equation*}
 G_\infty(x) = TG_\infty(x-\gamma)=T^2G_\infty(x),~x\in\Real.
\end{equation*}
Define
\begin{equation*}
 F_d(x)=G_\infty(x+\gamma/2),~x\in\Real.
\end{equation*}
It is easy to see that $F_d$ is a fixed point of the map $T$.

Now suppose that there exists another fixed point of $T$, call it $F_\dagger$. Define the shift function $\mathsf{sh}:\Real\To\Real$ such that
\begin{equation}\label{eq:twofixedptshift}
 F_\dagger(x)=F_d(x-\mathsf{sh}(x)).
\end{equation}
Assuming that $\mathsf{sh}$ is not constant, if we follow the steps of the proof of Lemma~\ref{lem:Tmapstrict}, replacing the relation $F(x)=TF(x-\widehat{F}(x))$ with (\ref{eq:twofixedptshift}), we will find that $\sup\mathsf{sh}<\sup\mathsf{sh},$ a contradiction. The only valid possibility is then to have $\mathsf{sh}(x)=\varphi$, a constant. But then, applying the $T$ map to $F_\dagger(x)=F_d(x-\varphi),~x\in\Real$, we get $F_\dagger(x)=F_d(x+\varphi),~x\in\Real$. Thus, $\varphi=0$, and we have a unique fixed point.

For any $F\in\mathcal{D}$, we have that $\widehat{T^{2k}F}(x)\To\gamma_F$ as $k\To\infty$ for all $x$. By Lipschitz continuity of $T^kF$, it can be verified that the sequence $T^{2k}F(x)$ is Cauchy, and hence convergent. It follows that
\begin{equation*}
 T^{2k}F(x)\xrightarrow{k\to\infty}F_d\parl{x-\frac{\gamma_F}{2}},\qquad  T^{2k+1}F(x)\xrightarrow{k\to\infty}F_d\parl{x+\frac{\gamma_F}{2}},
\end{equation*}
for all $x\in\Real$.
\end{proof}

\section*{Acknowledgments}
I thank Rajesh Sundaresan for helpful discussions and suggestions.
This work was supported by the Department of Science and Technology,
Government of India and by a TCS fellowship grant.

\bibliographystyle{imsart-number}
\bibliography{1}
\end{document}